\documentclass[a4paper,12pt]{amsart}

\usepackage{amssymb}
\usepackage{latexsym}
\usepackage{amsmath}
\usepackage{amscd}
\usepackage{graphicx}
\usepackage{palatino}

\title[Integral Euler characteristic of $\Out \, F_{11}$]
{Integral Euler characteristic of $\Out \, F_{11}$}

\author{Shigeyuki Morita}
\address{Graduate School of Mathematical Sciences, 
The University of Tokyo, 
3-8-1 Komaba, 
Meguro-ku, Tokyo, 153-8914, Japan}
\email{morita@ms.u-tokyo.ac.jp}

\author{Takuya Sakasai}
\address{Graduate School of Mathematical Sciences, 
The University of Tokyo, 
3-8-1 Komaba, 
Meguro-ku, Tokyo, 153-8914, Japan}
\email{sakasai@ms.u-tokyo.ac.jp}

\author{Masaaki Suzuki}
\address{Department of Frontier Media Science, 
Meiji University, 
4-21-1 Nakano, Nakano-ku, Tokyo, 164-8525, Japan}
\email{macky@fms.meiji.ac.jp}


\newtheorem{thm}{Theorem}[section]
\newtheorem{prop}[thm]{Proposition}

\theoremstyle{definition}

\newcommand{\hg}{\mathfrak{h}_{g,1}}

\newcommand{\hinf}{\mathfrak{h}_{\infty,1}}

\newcommand{\Ker}{\mathop{\mathrm{Ker}}\nolimits}

\newcommand{\Out}{\mathop{\mathrm{Out}}\nolimits}

\newcommand{\Q}{\mathbb{Q}}

\newcounter{fig}
\begin{document}

\maketitle

\begin{abstract}
We show that the integral Euler characteristic of the outer automorphism group of 
the free group of rank $11$ is $-1202$. 
%We show that the Euler characteristic of the weight 20 part of the
%symplectic derivation Lie algebra of the free Lie algebra is $-1299$.
%This is done by combining theoretical study and usage of computers.
%Our computational result together with Kontsevich's theorem on graph
%homologies implies that the integral Euler characteristic of
%the outer automorphism group of a free group of rank $11$ is $-1202$.
%It immediate follows that this outer automorphism group has
%at least $1203$ linearly independent odd dimensional
%non-trivial rational (co)homology classes.
\end{abstract}

\section{Introduction}\label{sec:intro}

This paper is a continuation of the previous paper \cite{mss3}, where we computed 
some parts of the Euler characteristics of three types of symplectic derivation Lie algebras. 

Let $\hg$ be the symplectic derivation Lie algebra of the free Lie algebra $\mathcal{L} (H)$ 
generated by the fundamental representation $H$ over $\mathbb{Q}$ of the symplectic 
group $\mathrm{Sp} (2g,\mathbb{Q})$. We may regard $H$ as a 
representation of the corresponding Lie algebra  $\mathfrak{sp} (2g,\mathbb{Q})$. 
Topologically, the vector space $H$ is the first rational homology group 
of a compact connected oriented surface of genus $g$ with one boundary component. 
The Lie algebra $\hinf$ obtained from $\hg$ by taking the direct limit with respect to $g$ 
is one of the three infinite dimensional Lie algebras considered 
by Kontsevich in \cite{kontsevich1, kontsevich2}. 
In these papers, he proved that, for $\hinf$ named the {\it Lie case}, 
the homology group of $\hinf$ is isomorphic to the free graded commutative algebra
generated by the stable homology group of $\mathfrak{sp}(2g,\Q)$ 
together with the totality of the cohomology groups of the
outer automorphism groups $\mathrm{Out}\,F_n$ of free groups $F_n$ of rank $n \ge 2$. 
This is done by a deep consideration on the relationship between 
the cell structure of the outer space given by Culler and Vogtmann \cite{cuv} 
and the chain complex which computes the graph homology associated with the Lie cyclic operad. 
The remaining two cases are named the {\it associative case} and the {\it commutative case}. 
Their homology groups are also related to other interesting geometrical objects such as 
cohomology groups of moduli spaces of Riemann surfaces and invariants of 
three dimensional manifolds etc.

Finding non-trivial rational (co)homology classes of $\mathrm{Out}\,F_n$ has been 
a difficult problem. 
A striking result by Galatius \cite{ga} shows that there are no rational {\it stable} reduced 
(co)homology classes. At present, only the three {\it unstable} classes, 
which are the first three of a series of classes introduced by the first author in \cite{morita99}, 
are shown to be non-trivial (see Conant and Vogtmann \cite{cov} and Gray \cite{gr}). 
Here we would like to mention that in a recent paper by Conant, Kassabov and 
Vogtmann \cite{ckv}, they constructed new homology classes of $\mathrm{Out}\,F_n$ 
which are related to the theory of elliptic modular forms, although their non-triviality is unknown. 
The situation being like this, it had been an important problem to prove the existence of non-trivial classes. 
In this context, our work in \cite{mss3} of determining some parts of the integral Euler characteristics 
showed that there exist at least one hundred new non-trivial classes and it also suggested that further 
computations should reveal the existence of many more classes. 
%This situation had shown that the existence of non-trivial classes is 
%an important problem. However, our computations in \cite{mss3} of some parts 
%of the integral Euler characteristics show that there exist at least one hundred 
%non-trivial classes and also suggest that it seems to exist much more.  

More specifically, the Lie algebra $\hg$ has a grading 
induced from that of the free Lie algebra $\mathcal{L}(H)=\oplus_{i=1}^\infty 
\mathcal{L}_i (H)$, 
so that we have a direct sum decomposition $\hg = \oplus_{k=0}^\infty \hg (k)$.  
Here $\hg (k)$ is the degree $k$ homogeneous part and 
$[\hg (k_1), \hg (k_2)] \subset \hg (k_1+k_2)$ holds for any $k_1$ and $k_2$. 
The explicit description of $\hg (k)$ is given by 
 \[ \hg (k)=\Ker \left( H \otimes \mathcal{L}_{k+1} (H) 
 \xrightarrow{[ \cdot, \cdot]} \mathcal{L}_{k+2} (H)\right).\]
Since the bracket operation $[\cdot, \cdot]$ on $\mathcal{L} (H)$ is equivariant with respect to 
the natural action of $\mathrm{Sp}(2g,\Q)$, the space $\hg (k)$ becomes 
an $\mathrm{Sp}(2g,\Q)$-module. 
It is known that the homology of the graded Lie algebra $\hg$ has another grading. 
That is, we have a direct sum decomposition  
\[H_\ast (\hg) = \bigoplus_{w=0}^\infty H_\ast (\hg)_w\]
with $H_\ast (\hg)_w$ obtained as the homology of the subcomplex generated by 
the chains in $\wedge^\ast \hg$ of total degree $w$ 
and called the {\it weight $w$ part} hereafter (see \cite[Section 2]{mss3}). 

Our main concern is the computation of $H_\ast (\hinf)_w$ after taking 
the direct limit with respect to $g$. 
At first glance, it may look too huge to handle. 
%It looks at first glance too huge to handle with. 
However, the following observation shows that it is not necessarily so. 
Let $\hg^+=\oplus_{k=1}^\infty \hg(k)$ be the ideal of the {\it positive} degree part. 
The spaces $H_\ast (\hg)_w$ and $H_\ast (\hg^+)_w$ are 
also $\mathrm{Sp}(2g,\Q)$-modules. 
As stated by Kontsevich \cite{kontsevich1} and proved in detail 
by Conant and Vogtmann \cite[Proposition 8]{cv0}  
(see also \cite[Section 2]{mss3}), we have 
\[H_\ast (\hg)_w=H_\ast (\hg^+)_w^{\mathrm{Sp}}\] 
for any $w \ge 1$. 
Here, for an $\mathrm{Sp}(2g,\Q)$-module $V$, we denote by $V^{\mathrm{Sp}}$ 
the invariant part for the $\mathrm{Sp}(2g,\Q)$-action. 
The general theory of $\mathrm{Sp}(2g,\Q)$-representations says that 
the invariant part $H_\ast (\hg^+)_w^{\mathrm{Sp}}$ as well as that of 
the corresponding chain complex stabilizes when 
$g$ becomes large. In particular, they are {\it finite} dimensional. 
Then Kontsevich's theorem says that the isomorphism 
\[PH_k (\hinf)_{2n} \cong H^{2n-k} (\Out \,F_{n+1};\Q)\]
holds for $n \ge 1$ and $k \ge 1$, 
where $PH_k (\hinf)_{2n}$ is the primitive part in $H_k (\hinf)_{2n}$ with respect to 
the commutative and co-commutative Hopf algebra structure (see \cite[Section 2]{cv0}). 

%However there have been known only a few results which deduce 
%new information about the graph homology and cohomology groups
%of  $\mathrm{Out}\,F_n$ or $\mathbf{M}_g^m/\mathfrak{S}_m$
%by making a direct use of the above theorem of Kontsevich.
%First, as for the Lie case, in \cite{morita99} the first named author defined a series of
%certain unstable homology classes of  $\mathrm{Out}\,F_n$
%by using his trace maps introduced in \cite{morita93}.
%Only the first three classes are known to be non-trivial (see Conant and Vogtmann \cite{cov} and Gray \cite{gr}).
%Second, recently Conant, Kassabov and Vogtmann \cite{ckv} made a remarkable
%new development in this direction and defined many more classes. 
%Thirdly, in the commutative case,
%the existence of two graph homology classes of {\it odd}
%degrees was proved, one in
%Gerlits \cite{ge} and the other in Conant, Gerlits and Vogtmann\cite{cgv}.
%Fourthly, as for the associative case,
%in \cite{morita08} a series of certain unstable homology classes for genus $1$ moduli
%spaces was introduced, all of which have been proved to be non-trivial by Conant \cite{con}. 
%Finally, in our recent paper \cite{mss1}
%we determined the stable abelianization of the Lie algebra in the associative case.
%As an application of this result, we obtained a new proof of an unpublished result of Harer.
%Church, Farb and Putman \cite{cfp1} gave a different proof.

In our previous paper \cite{mss3}, 
we determined the dimensions of the chain complex $C_i (\hinf^+)^\mathrm{Sp}$ which computes 
the $\mathrm{Sp}$-invariant homology of the Lie algebra $\hinf^+$ 
up to weight $18$. From this, 
we obtained the value of the Euler characteristic 
\[\chi(H_*(\mathfrak{h}^+_{\infty,1})^{\mathrm{Sp}}_{w})=\sum_{i=1}^w (-1)^i 
\dim \left(C_i (\hinf^+)_w^\mathrm{Sp}\right)\] 
of each weight $w \le 18$ summand. The result is given as follows: 

\begin{thm}[\cite{mss3}]\label{thm:chi}
The Euler characteristics $\chi(H_*(\mathfrak{h}^+_{\infty,1})_w^{\mathrm{Sp}})$ 
of the $\mathrm{Sp}$-invariant homology groups of 
the Lie algebra $\hinf^+$
up to weight $w \le 18$ are given by the following table:

\begin{center}
\begin{tabular}{|c||c|c|c|c|c|c|c|c|c|}
\hline
$w$ & $2$ & $4$ & $6$ & $8$ & $10$ & $12$ & $14$ & $16$ & $18$\\
\hline
$\chi(H_*(\mathfrak{h}^+_{\infty,1})^{\mathrm{Sp}}_{w})$ & 
$1$ & $2$ & $4$ & $6$ & $10$ & $16$ & $23$ & $13$ & $-96$ \\
\hline
\end{tabular}
\end{center}
\end{thm}
\noindent
Note that $C_i (\hinf^+)_w^\mathrm{Sp}$ is trivial 
if $w$ is odd. 
By combining Theorem \ref{thm:chi} with the description of the
generators of the stable cohomologies due to Kontsevich, we obtain the following result:  

\begin{thm}[\cite{mss3}]\label{thm:chip}
The integral Euler characteristics 
\[e(\mathrm{Out}\, F_{n})=
\displaystyle\sum_{i=0}^{2n-3} (-1)^i \dim 
\left(H^i (\mathrm{Out}\, F_n; \mathbb{Q})\right)\] 
of $\mathrm{Out}\, F_n$ up to $n \le 10$ are given as follows: 

\begin{center}
\begin{tabular}{|c||c|c|c|c|c|c|c|c|c|}
\hline
$n$ & $2$ & $3$ & $4$ & $5$ & $6$ & $7$ & $8$ & $9$ & $10$\\
\hline
$e(\mathrm{Out}\, F_n)$ & 
$1$ & $1$ & $2$ & $1$ & $2$ & $1$ & $1$ & $-21$ & $-124$ \\
\hline
\end{tabular}
\end{center}
\end{thm}
\noindent
By Theorem \ref{thm:chip}, the existence of non-trivial {\it odd} dimensional rational
cohomology classes of $\mathrm{Out}\, F_{n}$ was shown for the first time. 

The purpose of the present paper is to extend our results to the next step, weight $20$, 
by which we determine the integral Euler characteristic of $\mathrm{Out}\, F_{11}$. 
The details of our explicit computation will be given in Section \ref{sec:main}. 
We will also compare our results with the result of 
Smillie and Vogtmann \cite{sv} on the rational (or orbifold) Euler characteristics of 
$\mathrm{Out}\,F_n$. 

For the computations of this paper, we basically used the same methods as 
in \cite[Section 4]{mss3}, namely we made intensive use of well known software LiE 
%for representation theoretical computations 
and Mathematica. 
%, on which we ran our own programs. 
%One new equipment we used is 
%``TSUBAME 2.5'' supercomputer in the Tokyo Institute of Technology. 
%By running our Mathematica program on that machine, 
%we made a big database of 
%the irreducible decompositions of $GL$-representations as $\mathrm{Sp}$-representations. 
Since our computations heavily depend on computers, the checking process for the accuracy 
is as important as the actual computational process. 
Our checking methods were also discussed in \cite[Section 4]{mss3}. 

Finally, we comment about 
related works on the other two cases of Kontsevich's theorem. 
As for the commutative case, Willwacher and \v{Z}ivkovi\'{c} \cite{wz} 
recently obtained the generating function of 
the (total) Euler characteristic and computed the explicit values up to weight $60$. 
Our former results in \cite{mss3} are consistent with theirs. 
For the associative case, we can apply Gorsky's formula \cite{go2} for 
the equivariant Euler characteristics of moduli spaces of 
Riemann surfaces with marked points. 
The formula enables the authors to compute the Euler characteristics 
up to weight $250$ \cite{mss5}, which coincide with our former computation 
up to weight $16$ in \cite{mss3}. 
These facts would support the accuracy of our computations in this paper 
since many parts of the data on various symplectic modules 
we used are in common with those for the commutative and associative cases. 
On the other hand, no result is known about the generating function for the Lie case 
which is similar to Gorsky's formula for the associative case. 
There seem to exist difficulties peculiar to this case, 
which add an additional meaning to our computational results.
%These facts would support the accuracy of our computations in this paper 
%since the data on various symplectic modules we used 
%share many parts with those for the commutative and associative cases. 
%At present, there are no corresponding results on the generating function for the Lie case. 
%There seem to exist difficulties peculiar to this case, which give additional meaning of 
%our computational results.

{\it Acknowledgement} 
Some parts of the calculations of this paper were carried out 
on the TSUBAME 2.5 supercomputer in the Tokyo Institute of Technology. 
The authors are grateful 
to Professor Sadayoshi Kojima and Professor Mitsuhiko Takasawa 
who were very helpful for our computations in TSUBAME 2.5.

The authors were partially supported by KAKENHI 
(No.~24740040 and No.~24740035), 
Japan Society for the Promotion of Science, Japan.

\section{Main results}\label{sec:main}

We compute the dimension of the chain complex $C_i = C_i (\mathfrak{h}_{\infty,1}^+)^{\mathrm{Sp}}_{20}$ 
for $H_i (\mathfrak{h}_{\infty,1}^+)^{\mathrm{Sp}}_{20}$ explicitly. 
More precisely, we determine the dimension of the finite dimensional complex: 
\[
C_i = \bigoplus_{
\begin{subarray}{c}
i_1+\cdots + i_{20}=i\\ 
i_1+2 i_2+\cdots +20 i_{20}=20
\end{subarray}}
\left(
\wedge^{i_1} (\mathfrak{h}_{\infty,1}^+(1))\otimes 
\wedge^{i_2} (\mathfrak{h}_{\infty,1}^+(2))\otimes\cdots\otimes
\wedge^{i_{20}} (\mathfrak{h}_{\infty,1}^+(20)) \right)^{\mathrm{Sp}}
\]
where $\mathfrak{h}_{\infty,1}^+(k)$ is the degree $k$ part of $\mathfrak{h}_{\infty,1}^+$. 
The result is shown in Table \ref{tab:h}. 

\begin{table}[h]
\caption{The dimension of $C_i$} 
%\label{tab:im}
\begin{center}
\begin{tabular}{|c|r|}
%\noalign{\hrule height0.8pt}
%\hline
%$w$ & $2$ & $4$ & $6$ & $8$ & $10$ & $12$ & $14$ & $16$  & $18$ \\
\hline
& dimension \\
\hline
$C_1$ & $29729988$ \\
\hline
$C_2$ & $410769138$ \\
\hline
$C_3$ & $2864009351$ \\
\hline
$C_4$ & $13262053269$ \\
\hline
$C_5$ & $45353489325$ \\
\hline
$C_6$ & $120900142805$ \\
\hline
$C_7$ & $259222260499$ \\
\hline
$C_8$ & $455821729958$ \\
\hline
$C_9$ & $665350325867$ \\
\hline
$C_{10}$ & $811759271904$ \\
\hline
$C_{11}$ & $830129318093$ \\
\hline
$C_{12}$ & $711071098888$ \\
\hline
$C_{13}$ & $508080341074$ \\
\hline
$C_{14}$ & $300343387403$ \\
\hline
$C_{15}$ & $144874973588$ \\
\hline
$C_{16}$ & $55809757570$ \\
\hline
$C_{17}$ & $16607403485$ \\
\hline
$C_{18}$ & $3615255878$ \\
\hline
$C_{19}$ & $519201462$ \\
\hline
$C_{20}$ & $37584620$ \\
\hline
$\text{total}$ & $4946062104165$ \\
\hline
$\chi$ & $-1299$\\
\hline
%\noalign{\hrule height0.8pt}
\end{tabular}
\end{center}
\label{tab:h}
\end{table}

\begin{prop}[{\cite[Proposition 6.2]{mss3}}]
The weight generating function, denoted by $h(t)$, for the
$\mathrm{Sp}$-invariant stable homology group $H_*(\mathfrak{h}^+_{\infty,1})^{\mathrm{Sp}}$
is given by
$$
h(t)=\prod_{n=2}^\infty (1-t^{2n-2})^{- e(\mathrm{Out}\, F_n)}
$$
where $e(\mathrm{Out}\, F_n)$ denotes the {\it integral} Euler characteristic of $\mathrm{Out}\, F_n$.
\label{prop:ht}
\end{prop}

\begin{thm}\label{mainthm}
%The integral Euler characteristic of $\Out \, F_{11}$ is $-1202$.
\[ e(\mathrm{Out}\, F_{11}) = -1202. \]

\end{thm}

\begin{proof}
By Theorem \ref{thm:chi} and Table \ref{tab:h}, 
the weight generating function 
\[h(t)=\sum_{w=0}^\infty \chi(H_*(\mathfrak{h}^+_{\infty,1})^{\mathrm{Sp}}_w)t^w \]
%for $H_*(\mathfrak{h}^+_{\infty})^{\mathrm{Sp}}$ 
is written as 
\[h(t)=1+t^2+2t^4+4t^6+6t^8+10t^{10}+16t^{12}+23t^{14}+13t^{16}-96t^{18}-1299t^{20}
+\cdots.\]
By using the same method in \cite{mss3}, 
we can determine the Euler characteristics of the primitive parts, namely $e(\mathrm{Out}\, F_{11})$. 
To be precise, the Euler characteristic of lower terms in the weight $20$ is $-97$. 
Then the Euler characteristics of the primitive parts is $-1299 - (-97) = -1202$. 
In other words, if we consider 
\begin{align*}
\bar{h}(t)=(1-t^2)^{-1}(1-t^4)^{-1}&(1-t^6)^{-2}(1-t^8)^{-1}(1-t^{10})^{-2}\\
&(1-t^{12})^{-1}(1-t^{14})^{-1}(1-t^{16})^{21}(1-t^{18})^{124}(1-t^{20})^{1202},
\end{align*}
then $\bar{h}(t)$ is congruent to $h(t)$ modulo $t^{21}$. 
%$$
%h(t)-\bar{h}(t)\equiv 0 
% \mathrm{mod}\ 
%\bmod t^{21}.
%$$
By Proposition \ref{prop:ht}, we conclude that 
\[
 e(\mathrm{Out}\, F_{11})=-1202 . 
\]
\end{proof}

The fourth row of Table \ref{tab:hn} is the Euler characteristic of the primitive part 
which gives us the Euler characteristic of $\Out \, F_n$. 
\begin{table}[h]
\caption{$\text{Numbers of new generators for $H_*(\mathfrak{h}^+_{\infty,1})^{\mathrm{Sp}}_w$}$}
\begin{center}
\begin{tabular}{|c||r|r|r|r|r|r|r|r|r|r|}
\noalign{\hrule height0.8pt}
\hline
 $w$ & $2$ & $4$ & $6$ & $8$ & $10$ & $12$ & $14$ & $16$ & $18$ & $20$\\
\hline
$\chi$ & $1$ & $2$ & $4$ & $6$ & $10$ & $16$ & $23$ & $13$  & $-96$ & $-1299$\\
\hline
$\text{$\chi$ of lower terms}$ & $0$ & $1$ & $2$ & $5$ & $8$ & $15$ & $22$ & $34$  & $28$ & $-97$\\
\hline
$\text{$\chi$ of primitive part}$ & $1$ & $1$ & $2$ & $1$ & $2$ & $1$ & $1$ & $-21$  & $-124$ & $-1202$\\
\noalign{\hrule height0.8pt}
\end{tabular}
\end{center}
\label{tab:hn}
\end{table}

As a remarkable corollary, 
we see that there exist at least $1203$ {\it odd} dimensional non-trivial rational
cohomology classes of $\mathrm{Out}\,F_{11}$. 

By using this result, we can extend our former table in \cite{mss3} to Table \ref{tab:chie}, 
which compares the {\it rational} and the {\it integral} Euler characteristics of $\mathrm{Out}\,F_{n}$. 
The second row of Table \ref{tab:chie} is the rational Euler characteristics $\chi (\mathrm{Out}\,F_{n})$ 
of $\mathrm{Out}\,F_{n}$ given by Smillie and Vogtmann \cite{sv}, 
written up to the second decimal places here. 
The third row is the integral Euler characteristics $e(\mathrm{Out}\,F_{n})$ of $\mathrm{Out}\,F_{n}$ 
given by Theorems \ref{thm:chip} and \ref{mainthm}. 

Here we would like to mention the following two important open problems, which 
show a considerable difference between the Lie case and the other two cases. 
One is the asymptotic behavior of the rational Euler characteristics 
of $\mathrm{Out}\,F_{n}$. 
Smillie and Vogtmann \cite[Section 6]{sv} conjectured 
that the rational Euler characteristics of 
$\mathrm{Out}\,F_{n}$ are negative for all $n$, which holds for $n \le 100$ as mentioned in Vogtmann \cite{vogtmann}, 
and their absolute values grow exponentially with $n$. 
However, this conjecture is not settled yet. 
The other is the problem of determining whether the ratio of 
the rational Euler characteristics and the integral one tends to $1$ or not. 
In the case of the moduli spaces of Riemann surfaces, Harer and Zagier \cite{hz} 
proved that the ratio tends to $1$ asymptotically. 

\begin{table}[h]
\caption{$\chi$ versus $e$ for $\mathrm{Out}\,F_n$}
\begin{center}
\begin{tabular}{|c||c|c|c|c|c|c|c|c|c|c|}
\noalign{\hrule height0.8pt}
\hline
 $n$ & $2$ & $3$ & $4$ & $5$ & $6$ & $7$ & $8$ & $9$ & $10$ & $11$\\
\hline
$\chi$ & $-0.04$ & $-0.02$ & $-0.02$ & $-0.06$ & $-0.20$ & $-0.87$ & $-4.58$ & $-28.52$  & $-205.83$ & $-1690.70$\\
\hline
$e$ & $1$ & $1$ & $2$ & $1$ & $2$ & $1$ & $1$ & $-21$  & $-124$ & $-1202$\\
\noalign{\hrule height0.8pt}
\end{tabular}
\end{center}
\label{tab:chie}
\end{table}

\bibliographystyle{amsplain}

\end{document}